\documentclass{article}
\usepackage{amsthm}
\usepackage[margin=1in]{geometry}
\let\cedilla\c
\usepackage{etoolbox}
\robustify\cedilla

\usepackage{amssymb}
\usepackage{mathtools}
\usepackage{multirow}
\usepackage{booktabs}
\usepackage{enumitem}
\usepackage{xspace}
\usepackage{float}
\usepackage[linesnumbered,ruled]{algorithm2e}

\allowdisplaybreaks

\def\ackname{Acknowledgments}%
\def\acknowledgment{\par\addvspace{17pt}\small\rmfamily
	\trivlist\if!\ackname!\item[]\else
	\item[\hskip\labelsep
	{\bfseries\ackname}]\fi}

\newenvironment{acknowledgments}{\begin{acknowledgment}}
	{\end{acknowledgment}}

\DeclareRobustCommand{\rchi}{{\mathpalette\irchi\relax}}
\newcommand{\irchi}[2]{\raisebox{\depth}{$#1\chi$}}

\newcommand{\etal}{et~al.\@\xspace}%
\newcommand{\R}{\mathbb{R}}
\newcommand{\W}{\Omega}
\newcommand{\w}{\omega}
\newcommand{\N}{\mathcal{N}}
\renewcommand{\P}{\mathcal{P}}
\newcommand{\pathpizero}{\pathpi_0}
\newcommand{\pathpiplus}{\pathpi_+}
\newcommand{\pzero}{P_0}
\newcommand{\pplus}{P_+}
\newcommand{\rhop}[1]{\rho_P^{#1}}
\renewcommand{\c}{c}
\newcommand{\cc}{\sigma}
\newcommand{\HP}{H_P}
\newcommand{\HPstar}{H_{P^*}}
\newcommand{\F}{F}
\newcommand{\acoef}{\alpha}
\newcommand{\Acoef}{A}
\newcommand{\bconst}{\beta}
\newcommand{\phifunc}{\phi}
\newcommand{\psifunc}{\psi}
\newcommand{\deltaparam}{\eta}
\newcommand{\liftcoef}{\gamma}
\newcommand{\liftconst}{\delta}
\newcommand{\charvec}{\rchi^S}
\newcommand{\elm}{a}
\newcommand{\arc}{a'}
\newcommand{\ub}{u}
\newcommand{\pathpi}{\pi}
\newcommand{\pathfunc}{O}
\newcommand{\extfunc}{E}
\DeclareMathOperator{\conv}{conv}
\DeclareMathOperator{\proj}{proj}

\newcommand{\itab}[1]{\makebox[1.3cm]{#1\hfill}}{}

\usepackage[final]{changes}
\usepackage[normalem]{ulem}
\setaddedmarkup{\textcolor{blue}{\uline{#1}}}

\def\mystrut(#1,#2){\vrule height #1pt depth #2pt width 0pt} 

\usepackage{calc}
\newsavebox\CBox
\newcommand\hcancel[2][0.5pt]{%
	\ifmmode\sbox\CBox{$#2$}\else\sbox\CBox{#2}\fi%
	\makebox[0pt][l]{\usebox\CBox}%
	\textcolor{red}{\rule[0.5\ht\CBox-#1/2]{\wd\CBox}{#1}}}
\setdeletedmarkup{\hcancel[1pt]{#1}}

\newtheorem{theorem}{Theorem}
\newtheorem{lemma}{Lemma}
\newtheorem{definition}{Definition}
\newtheorem{proposition}{Proposition}

\begin{document}
\title{New Solution Approaches for the Maximum-Reliability Stochastic Network Interdiction Problem}




\author{Eli Towle\thanks{etowle@wisc.edu} \and James Luedtke\thanks{jim.luedtke@wisc.edu}}
\date{\small Department of Industrial and Systems Engineering, University of Wisconsin -- Madison}


\maketitle

\begin{abstract}
	We investigate methods to solve the maximum-reliability stochastic network interdiction problem (SNIP). In this
	problem, a defender interdicts arcs on a directed graph to minimize an attacker's probability of undetected traversal
	through the network. The attacker's origin and destination are unknown to the defender and assumed to be random. SNIP
	can be formulated as a stochastic mixed-integer program via a deterministic equivalent formulation (DEF). As the size of this 
	DEF makes it impractical for solving large instances, current approaches to solving SNIP rely on
	modifications of Benders decomposition.	We present two new approaches to solve SNIP. First, we introduce a new
	DEF that is significantly more compact than the standard DEF. Second, we propose
	a new path-based formulation of SNIP. The number of constraints required to define this formulation grows
	exponentially with the size of the network, but the model can be solved via delayed constraint generation.  We
	present valid inequalities for this path-based formulation which are dependent on the structure of the interdicted
	arc probabilities. We propose a branch-and-cut (BC) algorithm to solve this new SNIP formulation. Computational
	results demonstrate that directly solving the more compact SNIP formulation and this BC algorithm both provide an
	improvement over a state-of-the-art implementation of Benders decomposition for this problem.
	
	\noindent \raisebox{-0.5em}{\textbf{Keywords.} Network interdiction; stochastic programming; integer programming; valid inequalities}
\end{abstract}

\begin{acknowledgments}
	\begin{sloppypar}This work was supported by National Science Foundation grants CMMI-1130266 and SES-1422768.\end{sloppypar}
\end{acknowledgments}

\section{Introduction}

Network interdiction problems feature a defender and an attacker with opposing goals. The defender first modifies a network using a finite budget in an attempt to diminish the attacker's ability to perform a task. These modifications could include increasing arc traversal costs, reducing arc capacities, or removing arcs altogether. The attacker then optimizes his or her objective with respect to the newly modified network. Network interdiction problems have been studied in the context of nuclear weapons interdiction \cite{morton2007,pan2003}, disruption of nuclear weapons projects \cite{brown2009}, interdiction of drug smuggling routes \cite{wood1993}, and general critical infrastructure defense \cite{brown2009-2}.

\begin{sloppypar}In the maximum-reliability stochastic network interdiction problem (SNIP), formulated by Pan~\etal \cite{pan2003}, an attacker seeks to maximize the probability of avoiding detection while traveling through a directed network from an origin to a destination. There is a chance the attacker will be detected when traversing each arc. A defender may expend resources to place sensors on certain arcs, thereby increasing the probability of detecting an attacker on those arcs. The origin and destination of the attacker are unknown to the defender, and are assumed to be random. The defender's goal is to minimize the expected value of the attacker's maximum-reliability path through the network over all origin--destination scenarios. This can be interpreted as minimizing the overall probability of the attacker successfully attacking a critical infrastructure target or smuggling contraband to a target location undetected.\end{sloppypar}

Morton~\etal \cite{morton2007} present a deterministic equivalent formulation (DEF) of SNIP derived from the dual of the attacker's
optimization problem. Pan and Morton \cite{pan2008} improve the DEF by incorporating the values of
uninterdicted maximum-reliability paths into the model constraints. They also derive \replaced{valid inequalities}{SNIP-specific ``step inequalities''}
to strengthen the linear programming (LP) relaxation of the Benders master problem before initiating the Benders
algorithm. \deleted{These step inequalities are equivalent} \deleted{ to the mixing inequalities} \deleted{ of G{\"u}nl{\"u}k and Pochet
\cite{gunluk2001}.} Bodur~\etal \cite{bodur2016} investigate the strength of integrality-based cuts in conjunction with
Benders cuts for stochastic integer programs with continuous recourse, and use test instances of the SNIP problem to
demonstrate the value of integrality-based cuts within Benders decomposition.

We propose two new approaches to solve SNIP. First, we present a more compact DEF. This 
formulation combines constraints of the DEF for scenarios ending at the same destination, significantly 
reducing the number of constraints and variables in the DEF. Second, we present a new formulation of SNIP that includes constraints for
every origin--destination path. Although the formulation size grows exponentially with the number of arcs in the
network, the model can be solved with delayed constrained generation. A path's reliability function is a supermodular
set function representable with a strictly convex and decreasing function. We propose a branch-and-cut
algorithm that exploits this structure for each path through the network to derive valid inequalities. We consider three cases for the probabilities of
traversing interdicted arcs. Two of these cases use inequalities derived by Ahmed and Atamt{\"u}rk \cite{ahmed2011}
for the problem of maximizing a submodular utility function having similar structure.

\added{We conduct a computational study of the two proposed solution methods: directly solving the compact DEF, and
solving the path-based formulation with a branch-and-cut algorithm. We find the new methods have comparable performance,
and in particular they both outperform directly
solving the existing DEF and solving the existing DEF with a state-of-the art Benders branch-and-cut
method. Although the compact DEF is much easier to implement than the path-based branch-and-cut algorithm, it
cannot be extended to cases where the arc evasion probabilities depend on the origin/destination scenario. In contrast, the path-based decomposition can be extended to this case.} 

The maximum-reliability network interdiction problem is closely related to the shortest-path variant, in which the
network defender attempts to maximize the length of the attacker's shortest path from origin to destination.  When the
probabilities of traversing interdicted arcs in a maximum-reliability problem are all strictly positive \added{and the attacker's origin--destination pair is known}, the problem is
equivalent, via a logarithmic transformation, to a shortest-path network interdiction problem
\cite{ahuja1993,morton2010}. This transformation, however, is not valid when there exists an arc that reduces the attacker's
probability of successful traversal to zero when it is interdicted by the defender. A deterministic version of the shortest-path network interdiction problem was initially explored by
Fulkerson and Harding \cite{fulkerson1977}, and later by Golden \cite{golden1978}.

An alternative interdiction problem mostly unrelated to our work is the maximum-flow network interdiction problem,
introduced by Wollmer \cite{wollmer1964}. In this problem, a defender changes arc capacities in order to minimize the attacker's maximum flow. Deterministic and stochastic versions of
this problem have been studied by many authors, e.g., \cite{cormican1998,hemmecke2003,jeff2008,wood1993}.

In Section \ref{sec:existingresults}, we review relevant existing approaches to solving SNIP. We
propose a compact DEF of SNIP in Section \ref{sec:compact}. In Section \ref{sec:pathbased}, we propose a path-based formulation for SNIP and describe valid inequalities for the relevant mixed-integer set. We describe our computational experiments and results in Section \ref{sec:computational}.

\section{Problem statement and existing results}\label{sec:existingresults}
Let $N$ and $A$ denote the set of nodes and the set of arcs of a directed network. Let $D \subseteq A$ denote the set of arcs available for interdiction. In the first stage of SNIP, the defender may choose to install sensors on a subset of the interdictable arcs. The cost to install a sensor on arc $a \in D$ is $\c_a > 0$. The defender is constrained by installation budget $b > 0$. The probability of the attacker's undetected traversal through arc $a \in A$ without a sensor installed is $r_a \in (0,1]$. When a sensor is installed on arc $a \in D$, this probability reduces to $q_a \in [0, r_a)$. \added{The events of the attacker successfully traversing arcs are mutually independent.} \added{The attacker's origin and destination are unknown to the defender and assumed to be random.} Let $\W$ be the finite set of \added{the attacker's possible} origin--destination pairs, where $\w = (s, t) \in \W$ \added{is an origin--destination pair that} occurs with probability $p_\w > 0$ ($\sum_{\w \in \W} p_\w = 1$). We assume a path exists from $s$ to $t$ for all $(s, t) \in \W$. If not, the defender can discard that scenario from consideration. In the second stage of SNIP, the attacker's origin and destination are realized. The attacker traverses the network from the origin to the corresponding destination via the maximum-reliability path given the defender's set of interdicted arcs.

For $s,t \in N$, a simple $s$-$t$ path $P$ is a set of arcs, say $\added{\{}(i_0,i_1),(i_1,i_2),\ldots,$ $\replaced{(i_{|P|-1},i_{|P|})}{(i_{k-1},i_{k})}\added{\}}$,
where $i_0 = s$, $\replaced{i_{|P|}}{i_{k}} = t$, and the nodes $i_0,i_1,i_2,\ldots,i_{|P|} \in N$ are all distinct. Let $\P_{st}$ be the set of all simple paths from $s \in N$ to $t \in N$.  Let $S \subseteq D$ be the set of interdicted arcs, as chosen by the defender. The function
\begin{align*}
	h_P(S) \coloneqq \left( \prod_{a \in P} r_a \right) \left( \prod_{a \in P \cap S} \frac{q_a}{r_a} \right)
\end{align*}
calculates the probability of the attacker traversing the set of arcs $P$ undetected given the interdicted arcs $S$.

In the stochastic network interdiction problem (SNIP), the defender selects \replaced{arcs}{arc} to interdict to minimize the expected
value of the attacker's maximum-reliability path, subject to the budget restriction on the selected arcs. SNIP can be formulated as
	\begin{alignat}{3}
	\label{eq:path}
	\begin{aligned}
		\min_{S} && \ \smashoperator{\sum\limits_{\w = (s, t) \in \W}} p_{\w} & \max\{h_P(S) \colon P \in \P_{st}\} && \\
		\textrm{s.t.} && \sum\limits_{a \in S} \c_a & \leq b. &&
	\end{aligned}
\end{alignat}

\subsection{Deterministic equivalent formulation}\label{subsec:extensive}
A DEF of SNIP can be obtained by using first-stage binary variables $x_{a}$ to represent whether the defender installs a sensor on arc $a \in D$. That is, $x_a = 1$ if the defender elects to install a sensor on arc $a \in D$, and $x_a = 0$ if no sensor is installed on that arc. Second-stage continuous variables $\pi_i^\w$ represent the maximum probability of the attacker traveling undetected from $i \in N$ to $t$ in scenario $\w = (s, t) \in \W$.

The formulation uses second-stage constraints to calculate the $\pi$ variables for each scenario $\w \in \W$. The $\pi$ variables are calculated using an LP formulation of the dynamic programming (DP) optimality conditions for calculating a maximum-reliability path. Using the convention $q_a \coloneqq r_a$ and $x_a \coloneqq 0$ for all $a \in A \setminus D$, each $\pi_i^\w$ is calculated as
\begin{align}\label{eq:dp}
	\pi_i^\w & = \max_{a = (i,j) \in A} \left\lbrace \pi^\w_j\left[r_a + (q_a - r_a)x_a \right] \right\rbrace, \quad i \in N,\ \w \in \W.
\end{align}
The maximum probability of the attacker reaching $t$ undetected from node $i \in N$ is the maximum over all forward adjacent nodes $j \in N$ of $\replaced{\pi_j^{\w}}{\pi_j^{t}}$, multiplied by $r_a$ if the arc is not interdicted, or $q_a$ if it is interdicted.

The optimality conditions \eqref{eq:dp} imply the set of inequalities
\begin{align}\label{eq:dpineqs}
	\pi_i^\w & \geq \pi^\w_j\left[r_a + (q_a - r_a)x_a \right], \quad a = (i,j) \in A,\ \w \in \W.
\end{align}
The inequalities \eqref{eq:dpineqs} are nonlinear. Pan and Morton \cite{pan2008} derive the following DEF of SNIP that uses a linear
reformulation of these inequalities:
\begin{subequations}
	\label{eq:ext}
	\begin{alignat}{3}
		\min_{x, \pi}\enspace\ && \quad \smashoperator{\sum\limits_{\w = (s, t) \in \W}} p_\w \pi_{s}^\w &&& \label{eq:objfun} \\
		\textrm{s.t.}\enspace\ && \smashoperator{\sum\limits_{a \in D}} \c_a x_a & \leq b && \label{eq:extcon1} \\
		&& \pi_i^\w - r_a\pi_j^\w & \geq 0, && a=(i,j) \in A \setminus D,\ \w \in \W \label{eq:pi1} \\
		&& \pi_i^\w - r_a\pi_j^\w & \geq -(r_a - q_a) \ub_j^\w x_a, \qquad && a=(i,j) \in D,\ \w \in \W \label{eq:pi2} \\
		&& \pi_i^\w - q_a\pi_j^\w & \geq 0, && a=(i,j) \in D,\ \w \in \W \label{eq:pi3} \\
		&& \pi_t^\w & = 1, && \w = (s, t) \in \W \label{eq:pi4} \\
		&& x_a & \in \{0,1\}, && a \in D. \label{eq:extcon2}
	\end{alignat}
\end{subequations}

The objective function \eqref{eq:objfun} minimizes the expected value of the attacker's maximum-reliability path.
For each scenario $\w = (s, t) \in \W$, the parameter $\ub_j^{\w}$ represents the value of the maximum-reliability path
from $j \in N$ to $t$ when no sensors are installed, and hence is an upper bound on $\pi_j^{\w}$. These parameters are calculated in a model preprocessing step.
The linear constraints \eqref{eq:pi1}--\eqref{eq:pi3} formulate the nonlinear DP inequalities \eqref{eq:dpineqs}. Constraints \eqref{eq:pi1} enforce $\pi_i^\w \geq r_a
\pi_j^\w$ for all $a = (i, j) \in A \setminus D,\ \w \in \W$. For $\w \in \W$ and $a = (i, j) \in D$, if $x_a =
0$, then \eqref{eq:pi2} becomes $\pi_i^\w \geq r_a \pi_j^\w$, which dominates \eqref{eq:pi3}. On the other hand, if $x_a
= 1$, then \eqref{eq:pi3} implies $\pi_i^\w - r_a \pi_j^\w \geq -(r_a - q_a) \ub_j^\w$, 
\eqref{eq:pi3} dominates \eqref{eq:pi2}. Thus, in either case \eqref{eq:pi1}--\eqref{eq:pi3} are equivalent to
\eqref{eq:dpineqs}. Since the variables $\pi_s^{\w}$,
$\w=(s,t) \in \W$, have positive coefficients in the objective, \deleted{this implies} the equations \eqref{eq:dp} will be
satisfied \replaced{at an extreme-point}{in an} optimal solution. 

\subsection{Benders decomposition}\label{subsec:benders}
Directly solving the DEF \eqref{eq:ext} with a mixed-integer programming solver may be too time-consuming due to 
its large size. Benders decomposition can be used to decompose large problems like SNIP. After introducing the SNIP
formulation, Pan and Morton \cite{pan2008} outline a Benders decomposition algorithm for the SNIP DEF \eqref{eq:ext}. For a fixed vector $x \in [0, 1]^{\added{|D|}}$ satisfying $\sum_{a \in D} c_a x_a \leq b$, the $\pi_j^\w$ variables in \eqref{eq:ext} can be obtained by solving
\begin{subequations}
	\label{eq:extfixed}
	\begin{alignat}{3}
		\extfunc^{st}(x) \coloneqq \min_{\pi}\enspace && \pi_{s} \qquad\quad &&& \label{eq:fixedobjfun} \\
		\textrm{s.t.}\enspace && \pi_i - r_a\pi_j & \geq 0, && a=(i,j) \in A \setminus D \label{eq:fixedpi1} \\
		&& \pi_i - r_a\pi_j & \geq -(r_a - q_a) \ub_j^\w x_a, \qquad && a=(i,j) \in D \label{eq:fixedpi2} \\
		&& \pi_i - q_a\pi_j & \geq 0, && a=(i,j) \in D \label{eq:fixedpi3} \\
		&& \pi_t & = 1 && \label{eq:fixedpi4}
	\end{alignat}
\end{subequations}
for each scenario $\w = (s, t) \in \W$. The dual of \eqref{eq:extfixed} is
\begin{alignat}{3}
	\label{eq:dualsubproblem}
	\begin{aligned}
		\max_{y, z}\ && y_t - \smashoperator{\sum\limits_{a = (i, j) \in D}} (r_a - q_a) \ub_j^\w y_{ij} x_a \qquad \qquad\ &&& \\
		\textrm{s.t.}\ && \smashoperator{\sum\limits_{(s, j) \in A}} ( y_{s j} + z_{s j} ) & = 1 && \\
		&& \smashoperator{\sum\limits_{(i, j) \in A}} (y_{ij} + z_{ij}) - \smashoperator{\sum\limits_{a=(j, i) \in A}} (r_a y_{ji} + q_a z_{ji}) & = 0,\quad && i \in N \setminus \{s, t\} \\
		&& y_{t} - \smashoperator{\sum\limits_{a=(j, t) \in A}} (r_a y_{j t} + q_a z_{j t}) & = 0 && \\
		&& y_{ij}, z_{ij} & \geq 0, && (i, j) \in A \\
		&& y_{t} & \geq 0.
	\end{aligned}
\end{alignat}
Dual variables $y_{ij}$ correspond to DEF constraints \eqref{eq:fixedpi1} and \eqref{eq:fixedpi2}, $z_{ij}$ is the dual variable for constraints \eqref{eq:fixedpi3}, and $y_{t}$ is the dual variable for constraint \eqref{eq:fixedpi4}. We fix $z_{ij} \coloneqq 0$ for all $(i, j) \in A \setminus D$, as constraint \eqref{eq:fixedpi3} only applies to arcs $D$.

The Benders master problem is as follows:
\begin{subequations}
	\label{eq:benders}
	\begin{alignat}{3}
		\min_{x, \theta}\enspace && \sum\limits_{\w \in \W} p_\w \theta^\w &&& \\
		\textrm{s.t.}\enspace && \smashoperator{\sum\limits_{a \in D}} \c_a x_a & \leq b && \\
		&& \theta^\w \geq \bar{y}_{t}^{\w} &- \smashoperator{\sum\limits_{a = (i, j) \in D}} (r_a - q_a) \ub_j^\w \bar{y}_{ij}^{\w}x_a,\quad  && (\bar{y}^\w, \bar{z}^\w) \in K^{\w},\ \w = (s, t) \in \W \label{eq:optimalitycut} \\
		&& \theta^\w & \geq 0, && \w \in \W \\
		&& x_a & \in \{0, 1\}, && a \in D. \label{eq:bendersintegrality}
	\end{alignat}
\end{subequations}
The Benders algorithm begins by solving the master problem with no constraints of the form \eqref{eq:optimalitycut} to obtain a candidate solution $(\bar{x}, \bar{\theta})$. At iteration $k$ of the Benders cutting plane algorithm, a dual subproblem \eqref{eq:dualsubproblem} is solved for each scenario using the candidate solution $(\bar{x}, \bar{\theta})$ to find a dual-feasible extreme point $(\bar{y}^{\w}, \bar{z}^{\w})$. If a cut in the form of constraint \eqref{eq:optimalitycut} cuts off the candidate solution $\bar{\theta}^\w$, this cut is added to the Benders master problem by including the point in the set $K^\w$. The updated master problem is solved to generate a new candidate solution $(\bar{x}, \bar{\theta})$. This process is repeated until none of the scenario cuts constructed at an iteration of the algorithm cut off the candidate solution obtained by solving the master problem at the previous iteration.

Benders decomposition can also be used in a branch-and-cut algorithm. The integrality constraint
\eqref{eq:bendersintegrality} is relaxed and enforced within the branch-and-bound tree. At each integer-feasible
solution $(\bar{x}, \bar{\theta})$ obtained at a node in the master problem branch-and-bound tree, a dual
subproblem \eqref{eq:dualsubproblem} is solved for each scenario. Violated cuts are added to the LP formulation of that
node and it is re-solved. When no violated inequalities are found for an integer-feasible solution, the upper bound may be updated and
the node pruned.

Pan and Morton \cite{pan2008} enhance the Benders branch-and-cut algorithm by using optimal solutions to scenario subproblems from previous iterations to create step
inequalities for the relaxed Benders master problem. These additional inequalities are added to the formulation as cuts
to tighten the LP relaxation of the
master problem. This leads to a smaller branch-and-bound tree, which in turn reduces the time spent solving mixed-integer programs. 

The general-purpose Benders approach tested by Bodur~\etal \cite{bodur2016} is very effective in solving SNIP. A key
implementation detail of this approach is to first solve the relaxed Benders master problem,
obtained by removing the integrality restriction on $x$. Once no more Benders inequalities can be added to the LP
relaxation of the master problem, the integrality restriction on $x$ is restored, and the model, including the
identified Benders cuts, is passed to a mixed-integer programming solver to begin the Benders branch-and-cut algorithm. The solver adds its own general-purpose
integrality-based cuts to the formulation using the initial set of Benders cuts. This results in a stronger LP
relaxation bound, and reduces the size of the branch-and-bound tree. 

 Bodur~\etal \cite{bodur2016} employed a smaller gap tolerance as a stopping criterion ($10^{-3}$) for their Benders
 branch-and-cut implementation than Pan
and Morton  \cite{pan2008} used in their experiments ($10^{-2}$). We implemented the Benders
branch-and-cut algorithm as in \cite{bodur2016}, and tested it with a relative optimality gap tolerance of $10^{-2}$ to compare the
results to those published in \cite{pan2008}.
With these matching tolerances, the Benders
algorithm ran $3$--$4$ times faster than Pan and Morton's algorithm. Although the difference in hardware used in these
experiments makes it
impossible to directly compare the performance of these methods, we conclude that the implementation of Benders
branch-and-cut described in 
\deleted{in} \cite{bodur2016} is currently among the most efficient ways to solve instances of SNIP, and hence we compare against this approach in our numerical experiments.

\section{Compact deterministic equivalent formulation}\label{sec:compact}
For scenarios
sharing a destination node, the DP optimality conditions \eqref{eq:dp} for SNIP are identical. Hence, the 
DEF \eqref{eq:ext} contains redundancies by repeating the LP formulation of these optimality conditions,
constraints \eqref{eq:pi1}--\eqref{eq:pi4}, for each destination node.
Motivated by this observation, we present a compact DEF of SNIP that groups together second-stage
constraints for scenarios with a common destination. Let $T$ be the set of unique destination nodes: $T = \{ t \in N \colon (s,t) \in \W$ for some $s \in N \}$. 
Second-stage variables $\pi_j^t$ now represent the probability of the attacker
successfully traveling from node $j \in N$ to destination $t \in T$ in any scenario having destination $t$. We
obtain the following new DEF for SNIP:
\begin{subequations}
	\label{eq:com}
	\begin{alignat}{3}
		\min_{x, \pi}\ && \quad\ \smashoperator{\sum\limits_{\w = (s, t) \in \W}} \ p_\w \pi_{s}^{t} &&& \label{eq:comobj} \\
		\textrm{s.t.} && \smashoperator{\sum\limits_{a \in D}} \c_a x_a & \leq b && \label{eq:combudget} \\
		&& \pi_i^t - r_a \pi_j^t & \geq 0, && a = (i,j) \in A \setminus D,\ t \in T \label{eq:com1} \\
		&& \pi_i^t - r_a \pi_j^t & \geq -(r_a - q_a) \ub_j^t x_a, \qquad && a = (i,j) \in D,\ t \in T \label{eq:com2} \\
		&& \pi_i^t - q_a \pi_j^t & \geq 0, && a = (i,j) \in D,\ t \in T \label{eq:com3} \\
		&& \pi_t^t & = 1, && t\in T \label{eq:com4} \\
		&& x_a &\in \{0,1\}, && a \in D. \label{eq:comcon1}
	\end{alignat}
\end{subequations}
Parameter $\ub_j^t$ is the value of the attacker's maximum-reliability path from $j \in N$ to $t \in T$ when no sensors are installed. By definition, $\ub_j^t = \ub_j^\w$ for all $\w = (s, t) \in \W$ and $j \in N$.
The objective function \eqref{eq:comobj} weights each $\pi_s^t$ variable by its respective scenario probability. The summation of these terms represents the overall probability of a successful attack, which the defender seeks to minimize.

\added{In some SNIP formulations (e.g., \cite{dimitrov2011}), the evasion probabilities $r_a$ and $q_a$ depend on the scenario. If these probabilities depend on both $s$ and $t$ for each $(s,t) \in \Omega$, the compact DEF \eqref{eq:com} is not valid. However, if $r_a$ and $q_a$ depend only on $t$, \eqref{eq:com} remains valid. If the evasion probabilities depend only on $s$, we can derive a compact formulation of \eqref{eq:ext} by considering a network with all arc directions and origin--destination pairs reversed.}
\begin{proposition}
	\label{prop:1}
	 \deleted{Each feasible solution of \eqref{eq:ext} admits a feasible solution of \eqref{eq:com}} \deleted{ of equal objective value, and vice versa.} \added{For any feasible solution of the LP relaxation of \eqref{eq:ext}, there exists a feasible solution of the LP relaxation of \eqref{eq:com} of equal or lesser objective value, and vice versa.}
\end{proposition}
\begin{proof}
	\deleted{Let $(\bar{x}, \bar{\pi})$ be a solution to \eqref{eq:ext}.} \deleted{For all $t \in T$, select $s_t \in N$ such that} \deleted{$\w = (s_t, t) \in \W$.} \deleted{Let $\hat{\pi}^t_i \coloneqq \bar{\pi}^{\w}_i$ for all $i \in N$ and $t \in T$,} \deleted{and let $\hat{x}_{ij} \coloneqq \bar{x}_{ij}$ for all $(i, j) \in D$.} \deleted{$(\hat{x}, \hat{\pi})$ is feasible} \deleted{to \eqref{eq:com} with objective value} \deleted{$\sum_{\w = (s, t) \in \W} p_{\w} \hat{\pi}_{s}^{t} = \sum_{\w = (s, t) \in \W} p_{\w} \bar{\pi}_{s}^{\w}$.}
	
	\added{Let $(\bar{x}, \bar{\pi})$ be a solution to the LP relaxation of \eqref{eq:ext}. For $t \in T$, let $\mathcal{S}(t) \coloneqq \{s \in N \colon (s,t) \in \W\}$.} \added{By \eqref{eq:pi1}--\eqref{eq:pi4}, it holds that}
	\begin{alignat*}{2}
		\added{\min_{s \in \mathcal{S}(t)} \bar{\pi}_i^{st}} & \added{\geq \min_{s \in \mathcal{S}(t)} r_a \bar{\pi}_j^{st},} && \added{a = (i,j) \in A \setminus D,\ t \in T} \\
		\added{\min_{s \in \mathcal{S}(t)} \bar{\pi}_i^{st}} & \added{\geq \min_{s \in \mathcal{S}(t)} [r_a \bar{\pi}_j^{st} - (r_a - q_a) \ub_j^t \bar{x}_a], \qquad} && \added{a = (i,j) \in D,\ t \in T} \\
		\added{\min_{s \in \mathcal{S}(t)} \bar{\pi}_i^{st}} & \added{\geq \min_{s \in \mathcal{S}(t)} q_a \bar{\pi}_j^{st},} && \added{a = (i,j) \in D,\ t \in T} \\
		\added{\min_{s \in \mathcal{S}(t)} \bar{\pi}_t^{st}} & \added{= 1,} && \added{t \in T.}
	\end{alignat*}
	\added{Let $\hat{\pi}_i^t \coloneqq \min_{s \in \mathcal{S}(t)} \bar{\pi}_i^{st}$ for all $t \in T,\ i \in N$. Then, because
		$\min_{s \in \mathcal{S}(t)} r_a \bar{\pi}_j^{st} = r_a \min_{s \in \mathcal{S}(t)} \bar{\pi}_j^{st}$,
		$\min_{s \in \mathcal{S}(t)} q_a \bar{\pi}_j^{st} = q_a \min_{s \in \mathcal{S}(t)} \bar{\pi}_j^{st}$, and}
	\begin{align*}
		\added{\min_{s \in \mathcal{S}(t)} [r_a \bar{\pi}_j^{st} - (r_a - q_a) \ub_j^t \bar{x}_a]} &\added{= r_a \min_{s \in \mathcal{S}(t)} \bar{\pi}_j^{st} - (r_a - q_a) \ub_j^t \bar{x}_a,}
	\end{align*}
	$(\bar{x}, \hat{\pi})$ \added{is feasible to the LP relaxation of \eqref{eq:com}. $(\bar{x}, \hat{\pi})$ has objective value}
	\begin{align*}
		\added{\smashoperator{\sum_{\w = (s, t) \in \W}} \ p_{\w} \hat{\pi}_{s}^{t} = \enspace \smashoperator{ \sum_{\w = (s, t) \in \W}} \ p_{\w} \min_{s' \in \mathcal{S}(t)} \bar{\pi}_{s}^{s' t} \leq \enspace \smashoperator{ \sum_{\w = (s, t) \in \W}} \ p_{\w} \bar{\pi}_{s}^{st}.}
	\end{align*}
	Now, let $(\hat{x}, \hat{\pi})$ be a solution to the LP relaxation of \eqref{eq:com}. Let $\bar{\pi}_i^{\w} \coloneqq \hat{\pi}_i^t$ for all $i \in N$ and $\w = (s, t) \in \W$. \deleted{Let $\bar{x}_{ij} \coloneqq \hat{x}_{ij}$ for all $(i, j) \in D$.} $(\replaced{\hat{x}}{\bar{x}}, \bar{\pi})$ is a solution to \eqref{eq:ext} with objective value $\sum_{\w \added{= (s,t)} \in \W} p_{\w} \bar{\pi}_{s}^{\w} = \sum_{\w \added{= (s,t)} \in \W} p_{\w} \hat{\pi}_{s}^{t}$. 
\end{proof}
\added{It follows from Proposition~\ref{prop:1} that the optimal objective values of the LP relaxations of \eqref{eq:ext} and \eqref{eq:com} are equal, and also the optimal objective values of the original problems \eqref{eq:ext} and \eqref{eq:com} are equal.}

A Benders algorithm can be applied to the compact formulation by introducing variables $\theta^t$ to represent the
probability-weighted sum of maximum-reliability path values for scenarios ending at node $t \in T$.  We experimented
with a Benders decomposition on the DEF \eqref{eq:com} using the Benders algorithm of Bodur~\etal
\cite{bodur2016}. We found that the Benders algorithm performed much worse on this formulation than on the 
DEF \eqref{eq:ext}. The reason for this poor performance is that the root relaxation obtained after the
mixed-integer programming solver added its general-purpose cuts was much weaker when using a Benders reformulation of
\eqref{eq:com}, as compared to the Benders reformulation \eqref{eq:benders}. This is consistent with the results of
\cite{bodur2016}, which indicates that the integrality-based cuts derived in the formulation \eqref{eq:benders} can be
stronger than those derived in a ``projected'' Benders master problem that only uses $\theta^t$ variables.

\added{We also observed the Benders decomposition on the compact DEF \eqref{eq:com} performed worse than directly solving \eqref{eq:com}. Overall, there does not appear to be any advantage to implementing a Benders decomposition for the compact DEF.}

\section{Path-based formulation and cuts}\label{sec:pathbased}
We now derive a new mixed-integer linear programming formulation of SNIP based on the path-based formulation \eqref{eq:path}. With the introduction of binary variables $x_a$ to denote whether network arc $a \in D$ is interdicted, we map the set function $h_P$ to a vector function $\bar{h}_P$. In particular,
for $S \subseteq D$ we let $\charvec \in \{0, 1\}^{\added{|D|}}$ be the characteristic vector of the set $S$: $\charvec_a = 1$ if $a \in S$, and $0$ otherwise. 
Then we define $\bar{h}_P: \{0,1\}^{\added{|D|}} \rightarrow \mathbb{R}_{\added{+}}$ as 
\begin{align*}
	\bar{h}_P(x) \coloneqq \left[ \prod_{a \in P} r_a \right] \left[ \prod_{a \in P \cap D} \left( \frac{q_a}{r_a}
	\right)^{x_a} \right]
\end{align*}
so that $h_P(S) = \bar{h}_P(\charvec)$ for $S \subseteq D$.

Then model \eqref{eq:path} can be formulated as
\begin{alignat}{3}
	\label{eq:path2}
	\begin{aligned}
		\min_{x,\pathpi} && \smashoperator{\sum\limits_{\w = (s,t) \in \W}}\ p_\w & \pathpi_s^t && \\
		\textrm{s.t.} && \sum\limits_{a \in D} \c_a x_a & \leq b && \\
		&& \pathpi_s^t & \geq \max\{\bar{h}_P(x) \colon P \in \P_{st}\}, && (s, t) \in \W \\
		&& x_a & \in \{0, 1\}, && a \in D.
	\end{aligned}
\end{alignat}
\added{In contrast to the compact DEF \eqref{eq:com}, formulation \eqref{eq:path2} can be extended to the case where $r_a$ and $q_a$ depend on the scenario.}

We use the special structure of $h_P$ for each $P \in \P_{st}$ to build a linear formulation of model \eqref{eq:path2}.
\begin{definition}[E.g., \cite{schrijver2003}]
	$h \colon 2^\N \rightarrow \R$ is a \textit{supermodular} set function over a ground set $\N$ if $h(S_1 \cup \{\elm\}) - h(S_1) \leq h(S_2 \cup \{\elm\}) - h(S_2)$ for all $S_1 \subseteq S_2 \subseteq \N$ and $\elm \in \N \setminus S_2$.
\end{definition}
\begin{proposition}
	$h_P(\cdot)$ is a supermodular set function.
\end{proposition}
\begin{proof}
	Let $S_1, S_2 \subseteq D$ with $S_1 \subseteq S_2$. Let $\arc \in D \setminus S_2$. The result is immediate if $\arc \notin P$. Therefore, assume $\arc \in P$ to obtain
	\begin{alignat*}{2}
		h_P(S_1 \cup \{\arc\}) - h_P(S_1) &= \left( \prod_{a \in P} r_a \right) \left( \prod_{a \in P \cap S_1} \frac{q_a}{r_a} \right) \left(\frac{q_{\arc}}{r_{\arc}} - 1 \right) && \\
		&\leq \left( \prod_{a \in P} r_a \right) \left( \prod_{a \in P \cap S_2} \frac{q_a}{r_a} \right) \left(\frac{q_{\arc}}{r_{\arc}} - 1 \right) && \\
		&= h_P(S_2 \cup \{\arc\}) - h_P(S_2). &&
	\end{alignat*}
\end{proof}
The maximum of a set of supermodular functions is not in general a supermodular function, so $\max_{P \in \P_{st}}
h_P(S)$ is not necessarily supermodular. 
To exploit the supermodular structure for each individual path, we consider an equivalent formulation of \eqref{eq:path2} that contains an inequality for every scenario $(s, t) \in \W$ and every path from $s$ to $t$:
\begin{subequations}
	\label{eq:path3}
	\begin{alignat}{3}
		\min_{x,\pathpi} && \quad\enspace\ \smashoperator{\sum\limits_{\w = (s,t) \in \W}}\ p_\w &\pathpi_s^t && \\
		\textrm{s.t.} && \sum\limits_{a \in D} \c_a x_a & \leq b && \label{eq:path3budget} \\
		&& \pathpi_s^t & \geq \bar{h}_P(x), \qquad && P \in \P_{st},\ (s, t) \in \W \label{eq:path3path} \\
		&& x_a & \in \{0, 1\}, && a \in D.
	\end{alignat}
\end{subequations}
Each second-stage constraint includes a supermodular function on the right-hand side.

Consider the mixed-integer feasible region of \eqref{eq:path3} for a fixed scenario $(s, t) \in \W$ and path $P \in \P_{st}$\added{, without constraint \eqref{eq:path3budget}}:
\begin{align}
	\HP \coloneqq \left\lbrace (x,\pathpi) \in \{0,1\}^{\added{|D|}} \times \R \colon \pathpi \geq \bar{h}_P(x) \right\rbrace \label{eq:FP}.
\end{align}
We are interested in a linear formulation of $\HP$. Let $\rhop{a}(S) \coloneqq h_P(S \cup \{a\}) - h_P(S)$ be the marginal difference function of the set function $h_P$ with respect to arc $a \in D$. Nemhauser~\etal \cite{nemhauser1978} provide an exponential family of linear inequalities that can be used to define $\HP$. Applied to $\HP$, these inequalities are
\begin{align}
	\pathpi &\geq h_P(S) - \smashoperator{\sum\limits_{a \in S}} \rhop{a}(D \setminus \{a\})(1 - x_a) + \smashoperator{\sum\limits_{a \in D \setminus S}} \rhop{a}(S) x_a,\phantom{\emptyset} S \subseteq D \label{eq:ineq1} \\
	\pathpi &\geq h_P(S) - \smashoperator{\sum\limits_{a \in S}} \rhop{a}(S \setminus \{a\})(1 - x_a) + \smashoperator{\sum\limits_{a \in D \setminus S}} \rhop{a}(\emptyset) x_a,\phantom{D} S \subseteq D \label{eq:ineq2}.
\end{align}
Only one of the sets of inequalities \eqref{eq:ineq1} and \eqref{eq:ineq2} is required to define $\HP$
\cite{nemhauser1988}. Using these inequalities, the feasible region of $\HP$ can be formulated as
\begin{align*}
	\HP \coloneqq \left\lbrace (x,\pathpi) \in \{0,1\}^{\added{|D|}} \times \R \colon \eqref{eq:ineq1} \textrm{ or } \eqref{eq:ineq2} \right\rbrace.
\end{align*}
The number of inequalities required to define model \eqref{eq:path3} in this manner grows exponentially with the number
of arcs in a path, and the number of $s$--$t$ paths grows exponentially with the size of the network. Enumerating all
$s$--$t$ paths for scenario $(s, t) \in \W$ is impractical. Nevertheless, \eqref{eq:path3} lends itself to a delayed
constraint generation algorithm. Instead of adding inequalities for all possible paths, we add violated valid
inequalities as needed (see Section \ref{subsec:bncsummary}).

To demonstrate the potential of this formulation, we next show that if we could obtain
the convex hull of each set $\HP$, \replaced{we would have}{this would yield} a relaxation that is at least as strong as the LP relaxation of the
DEF \eqref{eq:ext} (and \eqref{eq:com}). \replaced{Although we may not be able to explicitly represent $\conv(\HP)$, this result}{This} implies that the path-based formulation has the potential
to yield a better LP relaxation if we can identify strong valid inequalities for $\conv(\HP)$.
\begin{theorem}\label{theorem:1}
	Consider the following LP relaxation of \eqref{eq:path3}:
	\begin{alignat}{3}
		\label{eq:path6}
		\begin{aligned}
			\min_{\added{x,} \pathpi}\ && \smashoperator{\sum\limits_{\w = (s,t) \in \W}} p_\w \replaced{\pathpi_s^t}{\pathpi_s^\w} &&& \\
			\emph{s.t.}\ && \sum\limits_{a \in D} \c_a x_a & \leq b && \\
			&& (x, \replaced{\pathpi_s^t}{\pathpi_s^\w}) & \in \conv(\HP), &&P \in \P_{st},\ \w = (s, t) \in \W \\
			&& x_a & \in [0, 1], && a \in D.
		\end{aligned}
	\end{alignat}
	For all $(\bar{x},\bar{\pathpi})$ feasible to \eqref{eq:path6}, there exists $\hat{\pi}$ such that
	$(\bar{x},\hat{\pi})$ is feasible to \eqref{eq:ext} and has objective value in \eqref{eq:ext} not greater than the
	objective value of $(\bar{x},\bar{\pathpi})$ in \eqref{eq:path6}.
\end{theorem}
The above theorem is a consequence of Lemma~\ref{lemma:1}. For a fixed $(s, t) \in \W$ and $x \in [0, 1]^{\added{|D|}}$ satisfying $\sum_{a \in D} \c_a x_a \leq b$, define
\begin{alignat*}{3}
	\pathfunc^{st}(x) \coloneqq \min_{\replaced{\pathpi_s}{\pathpi_s^t}}\ && \replaced{\pathpi_s}{\pathpi_s^t} \qquad &&& \\
	\textrm{s.t.}\ && (x, \replaced{\pathpi_s}{\pathpi_s^t}) & \in \conv(\HP),\ \enspace && P \in \P_{st}.
\end{alignat*}
Recall that $\extfunc^{st}(x)$ is defined in \eqref{eq:extfixed}.

\begin{lemma}\label{lemma:1}
	Let $\w = (s, t) \in \W$ and $\replaced{\bar{x}}{x} \in [0, 1]^{\added{|D|}}$ satisfy $\sum_{a \in D} \c_a \replaced{\bar{x}_a}{x_a} \leq b$. Then,
	\[ \extfunc^{st}(\replaced{\bar{x}}{x}) \leq \pathfunc^{st}(\replaced{\bar{x}}{x}). \]
\end{lemma}
\begin{proof}	
	The vector of all ones in $\R^{|N|}$ is feasible to \eqref{eq:extfixed}. A feasible solution to
	\eqref{eq:dualsubproblem} can be constructed as follows. Let $P = \{\added{(i_0,i_1),} (i_1,i_2), \deleted{(i_2,i_3),} \ldots, \replaced{(i_{|P|-1},i_{|P|})}{(i_{n-1},i_n)} \}
	\in \P_{st}$ be a simple $s$--$t$ path, where $\replaced{i_0}{i_1} = s$ and $\replaced{i_{|P|}}{i_n} = t$. Let $\replaced{\bar{y}_{i_0,i_1}}{\bar{y}_{i_1,i_2}} =
	1$. For $k = \replaced{1}{2}, \ldots, \replaced{|P|-1}{n}$, let $\bar{y}_{i_k, i_{k+1}} = r_{i_{k-1}, i_k} \bar{y}_{i_{k-1}, i_k}$. Let $\bar{y}_{ij} = 0$ for all $(i, j) \in A \setminus P$, and $\bar{z}_{ij} =
	0$ for all $(i,j) \in A$. Then $(\bar{y}, \bar{z})$ is feasible to \eqref{eq:dualsubproblem}. Because the feasible
	regions of \eqref{eq:extfixed} and \eqref{eq:dualsubproblem} are nonempty \deleted{and the objectives are bounded below by
	zero}, both problems have an optimal solution.
	Let $\pi^*$ be an optimal solution to \eqref{eq:extfixed}, with corresponding optimal dual solution $(y^*, z^*)$.	

	Assume first there exists an $s$--$t$ path $P^*$ such that one of the constraints \eqref{eq:fixedpi1}--\eqref{eq:fixedpi3} is
	satisfied at equality for all $a \in P^*$. Thus,
	\begin{align*}
			\pi^*_i = \max \left\lbrace r_a \pi^*_j - (r_a - q_a)u_j^{\w} \replaced{\bar{x}_a}{x_a},\ q_a \pi^*_j \right\rbrace
	\end{align*}
	for all $a=(i, j) \in P^*$.
	\added{For a fixed path $P \in \P_{st}$, let}    
	\begin{equation*}
		\added{G^{LP}(P) \coloneqq \left\lbrace \mystrut(44,44) \right. }
		\begin{alignedat}{2}
			\added{(x, \pi) \in [0,1]^{|D|}} &\added{\times \R^{|N|} \colon} && \\
			\added{\pi_i - r_a\pi_j} & \added{\geq 0,} && \added{a=(i,j) \in P \setminus D} \\
			\added{\pi_i - r_a\pi_j} & \added{\geq -(r_a - q_a) \ub_j^\w x_a, \quad} && \added{a=(i,j) \in P \cap D} \\
			\added{\pi_i - q_a\pi_j} & \added{\geq 0,} && \added{a=(i,j) \in P \cap D} \\
			\added{\pi_t} & \added{= 1} &&
		\end{alignedat}
		\added{\left. \mystrut(44,44) \right\rbrace. }
	\end{equation*}
	\added{Let $G^{IP}(P) \coloneqq \{(x, \pi) \in G^{LP}(P) \colon x \in \{0,1\}^{|D|} \}$. Removing constraints for arcs $a \in A \setminus P^*$ from the feasible region of \eqref{eq:extfixed} does not change the problem's optimal solution. Thus,}
	\begin{align*}
		\added{\extfunc^{st}(\bar{x})} &\added{= \min_{\pi} \{\pi_s \colon (\bar{x}, \pi) \in G^{LP}(P^*) \} } \\
			&= \min_{\pi_s} \{ \pi_s \colon (\bar{x}, \pi_s) \in \proj_{(x,\pi_s)} \big( G^{LP}(P^*) \big) \}.
	\end{align*}
	\added{For any $P \in \P_{st}$, $G^{IP}(P)$ is a formulation of $\{(x, \pi_s) \colon \pi_s \geq \bar{h}_P(x) \}$. That is, $\pi_s \geq \bar{h}_P(x)$ for any $(x,\pi) \in G^{IP}(P)$, and conversely, for any $x \in \{0,1\}^{|D|}$ and $\theta$ satisfying $\theta \geq \bar{h}_P(x)$, there exists some $\pi \in \R^{|N|}$ such that $\pi_s = \theta$ and $(x,\pi) \in G^{IP}(P)$. Thus,}
	\deleted{However, the optimal objective value} \deleted{of the resulting formulation, $\pi^*_s =
	\extfunc^{st}(x)$,} \deleted{is a lower bound on $\min\{\pathpi_s \colon (x, \pathpi_s) \in \conv(\HPstar)\}$,} \deleted{because
	$\conv(\HPstar)$ is a subset of the modified} \deleted{feasible region of $\eqref{eq:extfixed}$. Therefore,}
	\begin{align*}
		&\added{\proj_{(x,\pi_s)} G^{IP}(P^*) = \HPstar} \\
		\added{\implies} &\added{\proj_{(x,\pi_s)} G^{LP}(P^*) \supseteq \conv(\HPstar)} \\
		\added{\implies} &\extfunc^{st}(\replaced{\bar{x}}{x}) \leq \min\{\pathpi_s \colon (\replaced{\bar{x}}{x}, \pathpi_s) \in \conv(\HPstar)\} \leq \pathfunc^{st}(\replaced{\bar{x}}{x}).
	\end{align*}
	
	Finally, assume there is no $s$--$t$ path $P^* \in \P_{st}$ such that one of the constraints \eqref{eq:fixedpi1}--\eqref{eq:fixedpi3} is satisfied at equality for all $a \in P^*$. 
	Consider any $s$--$t$ path $P \in \P_{st}$. By assumption, there exists an arc $(i, j) \in P$ such that constraints
	\eqref{eq:fixedpi1}--\eqref{eq:fixedpi3} are not binding. By complementary slackness, the corresponding dual optimal
	$y^*_{ij}$ and $z^*_{ij}$ are $0$\replaced{. Then}{and} the flow of \replaced{any}{that} path through the \added{dual} network \added{of \eqref{eq:dualsubproblem}} is $0$\added{, and it must be the case that $\extfunc^{st}(\replaced{\bar{x}}{x}) = 0$}. \deleted{Because there exists no} \deleted{sequence of arcs through the dual network} \deleted{of \eqref{eq:dualsubproblem} that} \deleted{has positive flow, the dual optimal value} \deleted{is bounded above by $0$.} \deleted{Thus $E^{st}(x) \leq 0 \leq \pathfunc^{st}(x)$.} \added{Because $\pathfunc^{st}(x) \geq 0$, we have $\extfunc^{st}(\replaced{\bar{x}}{x}) \leq \pathfunc^{st}(\replaced{\bar{x}}{x})$.} 
	\end{proof}
	\begin{proof}[Proof of Theorem~\ref{theorem:1}]
		From Lemma~\ref{lemma:1}, we have 
		\[ \smashoperator{\sum_{\w = (s,t) \in \W}} p_\w \bar{\pathpi}_s^{\w} \geq \enspace \smashoperator{\sum_{\w = (s,t) \in \W}} p_\w \pathfunc^{st}(\bar{x}) \geq
		\enspace \smashoperator{\sum_{\w = (s,t) \in \W}} p_\w \extfunc^{st}(\bar{x}). \]
	This implies there exists $\hat{\pi} \in \R^{\added{|N| \times |\W|}}$ 
	such that $(\bar{x},\hat{\pi})$ is feasible to \eqref{eq:ext} with objective value $\sum_{\w = (s,t) \in \W} p_\w
	\hat{\pi}_s^\w \leq \sum_{\w = (s,t) \in \W} p_\w \bar{\pathpi}_s^\w$. 
\end{proof}

Computational experiments by Ahmed and \replaced{Atamt{\"u}rk}{Atam{\"u}rk} \cite{ahmed2011} show that the inequalities \eqref{eq:ineq1} and
\eqref{eq:ineq2} may provide a poor approximation of the set $\conv(\HP)$. 
Thus, in the following sections, we describe additional inequalities that can used to approximate $\conv(\HP)$ for all
$P \in \P_{st}$. We consider three different classes of inequalities, based on the traversal probabilities of interdicted arcs.

\subsection{Inequalities for the $q > 0$ case}\label{subsec:q>0}
In this section, we assume $q_a > 0$ for all $a \in D$. For a ground set $\N$, Ahmed and \replaced{Atamt{\"u}rk}{Atam{\"u}rk} \cite{ahmed2011} study valid inequalities for the mixed-integer set
\begin{align}
	\F = \left\lbrace (x, w) \in \{0,1\}^{\added{|\N|}} \times \R \colon w \leq f(\textstyle\sum_{i \in \N}\acoef_i x_i + \bconst) \right\rbrace, \label{eq:barF}
\end{align}
where $\acoef \in \R^{\added{|\N|}}_+$, $\bconst \in \R$, and $f$ is a strictly concave, increasing, differentiable function. They derive valid inequalities for the set $\F$ and prove that they dominate the submodular inequality equivalents of \eqref{eq:ineq1} and \eqref{eq:ineq2} applied to $\F$. These improved inequalities are shown empirically to yield significantly better relaxations than \eqref{eq:ineq1} and \eqref{eq:ineq2}.

We translate the set $\HP$ to match the structure of $\F$. Let $\N \coloneqq D$, $w \coloneqq -\pathpi$, $f(u) \coloneqq -\exp(-u)$, $\bconst \coloneqq -\sum_{a \in P} \log(r_a)$. For $a \in D$, let
\begin{align*}
	\acoef_a \coloneqq \begin{cases} \log(r_a) - \log(q_a) & \textrm{if } a \in P \cap D, \\ 0 & \textrm{otherwise.} \end{cases}
\end{align*}
Because $\log(r_a) > \log(q_a)$ for all $a \in D$, we have $\acoef \in \R^{\added{|D|}}_{+}$. With these definitions, $\HP$ is expressed in the form of $\F$.

We now describe how to calculate valid inequalities for $\HP$ using the results of Ahmed and Atamt{\"u}rk
\cite{ahmed2011}. \added{Consider a set of interdicted arcs $S \subseteq D$.} Without loss of generality, let $D \setminus S \coloneqq \{1, 2, \ldots, m\}$ be indexed such that
$\acoef_1 \geq \acoef_2 \geq \ldots \geq \acoef_m$, and let $\Acoef_k = \sum_{a = 1}^k \acoef_a$ for $k \in D
\setminus S$, with $\Acoef_0 = 0$. Also define $\acoef(S) = \sum_{a \in S} \acoef_a$. The subadditive lifting inequality
\begin{align}
		\pathpi \geq h_P(S) - \smashoperator{\sum_{a \in S}} \phifunc(-\acoef_a)(1 - x_a) + \smashoperator{\sum_{a \in D \setminus S}} \rhop{a}(S) x_a \label{eq:lifted1}
\end{align}
is valid for $\HP$ for any set of interdicted arcs $S \subseteq D$, where $\phifunc$ is calculated as follows \cite{ahmed2011}. Consider the function $\zeta \colon \R_- \rightarrow \R$, calculated according to Algorithm \ref{alg:1}.
\begin{algorithm}[ht]\label{alg:1}
	$k \leftarrow 0$\;
	\While{$k < m$ {\bf and} $\Acoef_k + \deltaparam < 0$}
	{
		$k \leftarrow k + 1$\;
	}
	$\zeta(\deltaparam) \leftarrow -\exp(-\acoef(S) - \Acoef_k - \bconst - \deltaparam) + \sum_{a = 1}^k \rhop{a}(S) + \exp(-\acoef(S) - \bconst)$\;
	\caption{Computing $\zeta(\deltaparam)$ \added{and $k$ given $\deltaparam$}}
\end{algorithm}
For a provided value of $\deltaparam$, Algorithm \ref{alg:1} \added{also} returns a specific $k$. With this $k$, let $\phifunc \colon \R_- \rightarrow \R$ be defined as
\begin{align*}
	\phifunc(\deltaparam) \coloneqq \begin{cases} \zeta(\mu_k - \Acoef_{k-1}) + \rhop{k}(S) \frac{b_k(\deltaparam)}{\acoef_k} & \textrm{if }\mu_k - \Acoef_k \leq \deltaparam \leq \mu_k - \Acoef_{k-1}, \\ \zeta(\deltaparam) & \textrm{otherwise,} \end{cases}
\end{align*}
where $\mu_k = -\log(-\rhop{k}(S)/\acoef_k) - \acoef(S) - \bconst$ and $b_k(\deltaparam) = \mu_k - \Acoef_{k-1} - \deltaparam$. Inequality \eqref{eq:lifted1} dominates the general supermodular inequality \eqref{eq:ineq1} \cite{ahmed2011}.

We now construct Ahmed and Atamt{\"u}rk's inequalities that dominate the supermodular inequalities \eqref{eq:ineq2}. Let
$S = \{1, 2, \ldots, n\}$ be indexed such that $\acoef_1 \geq \acoef_2 \geq \ldots \geq \acoef_n$. Also, let
$\Acoef_k = \sum_{a = 1}^k \acoef_a$ for $k \in S$, with $\Acoef_0 = 0$. We define the function $\xi \colon \R_+
\rightarrow \R$ to be calculated according to Algorithm \ref{alg:2}.

\begin{algorithm}[ht]\label{alg:2}
	$k \leftarrow 0$\;
	\While{$k < \replaced{n}{r}$ {\bf and} $\Acoef_k < \deltaparam$}
	{
		$k \leftarrow k + 1$\;
	}
	$\xi(\deltaparam) \leftarrow -\exp(-\acoef(S) + \Acoef_k - \bconst - \deltaparam) - \sum_{a = 1}^k \rhop{a}(S \setminus \{a\}) + \exp(-\acoef(S) - \bconst)$\;
	\caption{Computing $\xi(\deltaparam)$ \added{and $k$ given $\deltaparam$}}
\end{algorithm}
\replaced{Algorithm \ref{alg:2} also calculates a $k$ from the given $\deltaparam$. With this $k$,}{With the $\deltaparam$-dependent $k$ calculated in Algorithm \ref{alg:2},} let
\begin{align*}
	\psifunc(\deltaparam) \coloneqq \begin{cases} \xi(\Acoef_k - \nu_k) + \rhop{k}(S \setminus \{k\}) \frac{b_k(\deltaparam)}{\acoef_k} & \textrm{if }\Acoef_{k-1} - \nu_k \leq \deltaparam \leq \Acoef_k - \nu_k, \\ \xi(\deltaparam) & \textrm{otherwise,} \end{cases}
\end{align*}
where $\nu_k = \acoef(S) + \log(-\rhop{k}(S \setminus \{k\}) / \acoef_k) - \bconst$ and $b_k(\deltaparam) = \Acoef_k - \nu_k - \deltaparam$ for every $k \in \{1, 2, \ldots, r\}$. For any $S \subseteq D$, the inequality
\begin{align}
	\pathpi \geq h_P(S) - \smashoperator{\sum_{a \in S}} \rhop{a}(S \setminus \{a\}) (1 - x_a) - \smashoperator{\sum_{a \in D \setminus S}} \psifunc(\acoef_a)x_a \label{eq:lifted2}
\end{align}
is valid for $\HP$. For all $a \in D \setminus P$, the coefficients in inequalities \eqref{eq:lifted1} and \eqref{eq:lifted2} are $0$.

We next discuss separation of the inequalities \eqref{eq:lifted1} and \eqref{eq:lifted2} for use in a delayed
constraint generation framework.
Given a point $(\bar{x}, \bar{\pathpi}) \in [0, 1]^{\added{|D|}} \times \R$, we want to determine if there is a set $S$ such that $(\bar{x},\bar{\pathpi})$ violates either \eqref{eq:lifted1} or \eqref{eq:lifted2}. For an integral vector $\bar{x}$, we construct the set $S = \{a \in D \colon \bar{x}_a = 1\}$. If $\bar{\pathpi} < h_P(S)$, then inequalities \eqref{eq:lifted1} and \eqref{eq:lifted2} are violated and can be added to the formulation of $\HP$ to cut off $(\bar{x}, \bar{\pathpi})$.
	
If $\bar{x}$ is not integral, we use Ahmed and Atamt{\"u}rk's scheme to heuristically construct a set $S$ using
$\bar{x}$. This set is used in inequalities \eqref{eq:lifted1} and \eqref{eq:lifted2} to cut off $(\bar{x},
\bar{\pathpi})$, if possible. We first consider constructing a set $S$ to apply inequality \eqref{eq:lifted1}. We first solve the following nonlinear program:
\begin{align}
	\max_{z \in [0,1]^{\added{|D|}}}\ \frac{-\bar{\pathpi} + 1}{ -\exp(-\textstyle\sum_{a \in D}\acoef_a z_a - \bconst) + 1} + \smashoperator{\sum\limits_{a \in D}} \frac{\rhop{a}(\emptyset)}{h_P(\emptyset) + 1} \bar{x}_a (1 - z_a). \label{eq:sep1}
\end{align}
Given a solution $\bar{z}$, we greedily round the fractional components in $\bar{z}$ that result in the least reduction
in objective value when rounded to either $0$ or $1$. Given the rounded solution, say $\bar{z}'$, we set $S = \{a \in D
\colon \bar{z}'_a = 1\}$. If $\pathpi < h_P(S) - \sum_{a \in S} \phifunc(-\acoef_a)(1 - \bar{x}_a) + \sum_{a \in D
\setminus S} \rhop{a}(S) \bar{x}_a$, valid inequality \eqref{eq:lifted1} can be added to the relaxation to cut off $(\bar{x}, \bar{\pathpi})$.

We solve a related nonlinear program to find a set $S$ to apply inequality \eqref{eq:lifted2}:
\begin{align}
	\max_{z \in [0,1]^{\added{|D|}}}\ \frac{-\bar{\pathpi} + 1}{ -\exp(-\textstyle\sum_{a \in D}\acoef_a z_a - \bconst) + 1} - \smashoperator{\sum\limits_{a \in D}} \frac{\rhop{a}(\emptyset)}{h_P(\{a\}) + 1} (1 - \bar{x}_a) z_a. \label{eq:sep2}
\end{align}
We again greedily round the solution of \eqref{eq:sep1} to obtain the characteristic vector of $S$. If $\pathpi < h_P(S) - \sum_{a \in S} \rhop{a}(S \setminus \{a\}) (1 - \bar{x}_a) - \sum_{a \in D \setminus S} \psifunc(\acoef_a) \bar{x}_a$, inequality \eqref{eq:lifted2} cuts off $(\bar{x}, \bar{\pathpi})$.

Yu and Ahmed \cite{yu2017} consider deriving stronger valid inequalities by imposing a knapsack constraint on the set $\F$ given in equation \eqref{eq:barF}. In particular,
\begin{align*}
	\F = \left\lbrace (x, w) \in \{0,1\}^{\added{|\N|}} \times \R \colon w \leq f(\textstyle\sum_{i \in \N}\acoef_i x_i + \bconst),\ \sum_{i \in \N} x_i \leq k \right\rbrace,
\end{align*}
where $k > 0$. Although our test instances include this structure, we do not explore applying their
results here. In particular, when interdicting arcs along an {\it individual} path, the knapsack constraint is only 
relevant if the number of arcs in the path under consideration is larger than the defender's budget. 

\subsection{Inequalities for $q = 0$ case}\label{subsec:q=0}
When $q_a = 0$ for all $a \in D$, $\log(q_a)$ is no longer defined, and the results of Section \ref{subsec:q>0} do not apply. We consider a different class of inequalities for this special case of $q$.

If all interdicted arc probabilities are $0$, $\bar{h}_P(x)$ simply reduces to
\begin{align*}
	\bar{h}_P(x) = \left[\, \prod_{a \in P} r_a \right] \left[ \enspace\ \smashoperator{\prod_{a \in P \cap D}}\, (1 - x_a) \right].
\end{align*}
If any arc $a \in P \cap D$ is interdicted, $\bar{h}_P(x)$ equals $0$. By setting $\hat{\pathpi} \coloneqq \pathpi / \prod_{a \in P} r_a$, our relevant mixed-integer linear set is
\begin{align*}
	\HP &= \left\lbrace (x,\hat{\pathpi}) \in \{0,1\}^{\added{|D|}} \times \R \colon \hat{\pathpi} \geq \textstyle\prod_{a \in P \cap D} (1 - x_a) \right\rbrace.
\end{align*}
We are again interested in valid inequalities for $\HP$.  The inequalities
\begin{align}
	\hat{\pathpi} &\geq 1 - \smashoperator{\sum_{a \in P \cap D}} x_a \label{eq:french1}
\end{align}
and $\hat{\pathpi} \geq 0$ are valid for $\HP$ \cite{fortet1960} and define its convex hull \cite{alkhayyal1983}.

\subsection{Inequalities for the mixed case}\label{subsec:qmixed}
We now consider the case where $q_a > 0$ for some, but not all, $a \in D$. 
Let $\pplus \coloneqq \{a \in P \cap D \colon q_a > 0\}$ and $\pzero \coloneqq \{a \in P \cap D \colon q_a = 0\}$.
\replaced{For any $U \subseteq A$, let $r(U) \coloneqq \prod_{a \in U} r_a$.}{For $S \subseteq A$ let $r(S) \coloneqq \prod_{a \in S} r_a$.} We write $\bar{h}_P(x)$ as
\begin{align}
	\bar{h}_P(x) &= r(P \setminus D) \bar{h}_{\pplus}(x) \bar{h}_{\pzero}(x). \label{hprewrite}
\end{align}
We introduce two new variables, $\pathpiplus \in \left[0, r(\pplus) \right]$ and $\pathpizero \in [0, r(\pzero)]$ to
represent $\bar{h}_{\pplus}(x)$ and $\bar{h}_{\pzero}(x)$, respectively, and arrive at the formulation:
\begin{align}
	\pathpi &\geq r(P \setminus D) \pathpiplus \pathpizero \label{eq:theta1} \\
	\pathpiplus &\geq \bar{h}_{\pplus}(x) \label{eq:thetaplus1} \\
	\pathpizero &\geq \bar{h}_{\pzero}(x). \label{eq:thetazero1}
\end{align}
Our approach is to use results from Sections \ref{subsec:q>0} and \ref{subsec:q=0} to derive valid inequalities for
\eqref{eq:thetaplus1} and \eqref{eq:thetazero1}, respectively, and 
relax the nonconvex constraint \eqref{eq:theta1} using the  McCormick inequalities \cite{mccormick1976}, which in this
case reduce to
\begin{align}
	\pathpi & \geq r(P \setminus D) \left[ r(\pzero)\pathpiplus - r(\pplus)(r(\pzero) - \pathpizero) \right] \label{eq:q0ineq1}
\end{align}
and $\pathpi \geq 0$. 

Let $\liftcoef^S \in \R^{\added{|\pplus|}}$ and constant $\liftconst^S \in \R$ be coefficients from one inequality of the form
\eqref{eq:lifted1} (with $P_+$ taking the place of $P$) for some $S \subseteq \replaced{D}{P_+}$. An inequality of the form \eqref{eq:lifted2} may also be used. We relax \eqref{eq:thetaplus1} to the constraint
\begin{align}
	\pathpiplus & \geq \smashoperator{\sum_{a \in \pplus}} \liftcoef^S_a x_a + \liftconst^S. \label{eq:fme1}
\end{align}
Inequality \eqref{eq:thetazero1} is relaxed as $\pathpizero \geq 0$, and, applying \eqref{eq:french1}, 
\begin{align}
	\pathpizero & \geq r(\pzero) \left( 1 - \smashoperator{\sum_{a \in \pzero}} x_a \right) .  \label{eq:fme2}
\end{align}
Finally, we project variables $\pathpiplus$ and $\pathpizero$ out of inequality \eqref{eq:q0ineq1} using Fourier-Motzkin elimination with constraints \eqref{eq:fme1}--\eqref{eq:fme2} and $\pathpizero \geq 0$. This results in the two inequalities
\begin{align}
	\pathpi & \geq r(P \setminus D) r(\pzero) \left[\ \smashoperator{\sum_{a \in \pplus}} \liftcoef^S_a x_a +
	\liftconst^S - r(\pplus)\smashoperator{\sum_{a \in \pzero}} x_a \right] \label{eq:q01} \\
	\pathpi & \geq r(P \setminus D) r(\pzero) \left[\ \smashoperator{\sum_{a \in \pplus}} \liftcoef^S_a x_a +
	\liftconst^S - r(\pplus) \right]. \label{eq:q02}
\end{align}
Inequality \eqref{eq:q02} is dominated by $\pathpi \geq 0$, because $\sum_{a \in \pplus} \liftcoef^S_a x_a +
\liftconst^S
\leq \pathpiplus \leq r(P_+)$. Therefore, inequality \eqref{eq:q01} is the only inequality we obtain from this
procedure for the given inequality \eqref{eq:lifted1} defined by $S$. 

We next argue that for any integer solution $(\bar{x},\bar{\pathpi}) \in \{0, 1\}^{\added{|D|}} \times \R_+$, if it is not feasible
(i.e., $(\bar{x},\bar{\pathpi}) \notin \HP$),
then we can efficiently find an inequality of the form \eqref{eq:q01} that this solution violates. Indeed, 
first observe that if $\bar{x}_a = 1$ for any $a \in P_0$ then $\bar{h}_P(\bar{x}) = 0 \leq \bar{\pathpi}$, by
assumption, and hence the solution is feasible. So, we may assume $\bar{x}_a = 0$ for all $a \in P_0$. We set $S =
\{ a \in P_+ : \bar{x}_a = 1\}$, which yields 
$\sum_{a \in \pplus} \liftcoef^S_a \bar{x}_a + \liftconst^S = h_{\pplus}(S) = \bar{h}_{\pplus}(\bar{x})$. 
Inequality \eqref{eq:q01} thus yields
\[ \pathpi \geq   r(P \setminus D) r(\pzero)  \left[ \bar{h}_{\pplus}(\bar{x}) - 0 \right] = \bar{h}_P(\bar{x})  \]
by \eqref{hprewrite} since $\bar{h}_{\pzero}(\bar{x}) = r(\pzero)$.


\subsection{Path-based branch-and-cut algorithm}\label{subsec:bncsummary}

We next describe how the inequalities derived in Sections \ref{subsec:q>0}--\ref{subsec:qmixed}  can be used within a
branch-and-cut algorithm to solve the path-based formulation \eqref{eq:path3}. 
Similar to the use of Benders decomposition in a branch-and-cut algorithm described in Section \ref{subsec:benders}, the algorithm is
based on solving a master problem via branch-and-cut in which the constraints \eqref{eq:path3path} are relaxed and approximated with
cuts. 


Let $(\bar{x}, \bar{\pathpi})$ be a solution obtained in the branch-and-cut algorithm with integral $\bar{x}$.
Since \eqref{eq:path3path} are relaxed, we must
check if this solution satisfies these constraints, and if not, finding a cut that this solution violates.
We assign arc traversal probabilities $\cc_a$ for $a \in D$ according to the formula
\begin{align}
	\cc_a &= r_a^{1 - \bar{x}_a} q_a^{\bar{x}_a}. \label{reldef1}
\end{align}
For each $\w = (s,t) \in \W$, we find a maximum-reliability path $\bar{P} \in \P_{st}$. If $\bar{\pathpi}_s^t \geq
\bar{h}_{\bar{P}}(\bar{x})$ then this solution is feasible to \eqref{eq:path3path} for this $\w$, since, by
construction, $\bar{h}_{\bar{P}}(\bar{x}) = \max_{P \in \P_{st}} \bar{h}_P(\bar{x})$. Otherwise, depending on if $q_a =
0$ for any or all $a \in \bar{P}$, we derive a cut from one of Sections \ref{subsec:q>0}--\ref{subsec:qmixed}, using
this path $\bar{P}$ and scenario $\w$. By construction, for the solution $\bar{x}$, this cut enforces that $\pi_s^t
\geq \bar{h}_{\bar{P}}(\bar{x})$, and hence (i) cuts off the current infeasible solution $(\bar{x}, \bar{\pathpi})$, and
(ii) implies that any solution $(x,\pathpi)$ of the updated master problem with $x =\bar{x}$ will satisfy
\eqref{eq:path3path} for this $\w$. This ensures
that the branch-and-cut algorithm is finite (since at most finitely many such cuts are needed at finitely many integer solutions) and
correct (since any infeasible solution obtained in the algorithm will be cut off).

One may also attempt to generate cuts at solutions $(\bar{x}, \bar{\pathpi})$ in which $\bar{x}$ is not necessarily
integer, in order to improve the LP relaxation.
To find a path for a scenario $\w = (s,t)$ in this case, we again use formula \eqref{reldef1} to define arc
reliabilities, and then find the most reliable $s$-$t$ path. With this path, we again apply the methods in Sections
\ref{subsec:q>0}--\ref{subsec:qmixed}  to attempt to derive a violated cut. When $\bar{x}$ is not integral, this
approach is not guaranteed to find a violated cut, even if one exists. Thus, to increase the chances a cut is found, we
may also consider identifying another path by assigning arc costs as follows:
\begin{align*}
	\cc_a &= (1 - \bar{x}_a)r_a +  \bar{x}_a q_a,
\end{align*}
and then finding a maximum-reliability path using these arc reliabilities.
Our preliminary tests did not reveal any obvious benefit of using one particular arc-reliability calculation method. 


In our implementation, before starting the branch-and-cut algorithm we solve a relaxation of the master problem \eqref{eq:path3} in which constraints \eqref{eq:path3path}
are dropped, and the integrality constraints are relaxed. After solving this LP relaxation, 
we attempt to identify violated cuts for the relaxation solution, and if found, we add them to the master relaxation and
repeat this process until no more violated cuts are found.
We then begin the branch-and-cut process with all cuts found when solving the LP relaxation included in the mixed-integer programming
formulation. This allows the solver to use these cuts to generate new cuts in the master problem relaxation. At any point in the
branch-and-bound process where a solution $(\bar{x},\bar{\pathpi})$ with $\bar{x}$ integral is obtained, we identify if
there is any violated cut (as discussed above) and if so add it to the formulation via the lazy constraint callback function. 

\section{Computational experiments}\label{sec:computational}
We test the different solution methods on SNIP instances from Pan and
Morton \cite{pan2008}, which consist of a network of $783$ nodes and $2586$ arcs, $320$ of which
can be interdicted by the defender. Each instance considers the same $456$ scenarios. 
We consider three cases for the $q$ vector outlined by Pan and
Morton \cite{pan2008}. In particular, we consider instances with $q \coloneqq 0.5r$, $q \coloneqq 0.1r$, and $q
\coloneqq 0$. Pan and Morton also considered $q_a$ values independently sampled from a $U(0,0.5)$ distribution. 
Since the DEF for most of these instances could be solved within $30$ seconds with little or no branching, we exclude
these in our experiments. The cost of interdiction is $\c_a = 1$ for all $a \in D$. Seven different budget levels were
tested for each network and $q$ level, and there are five test instances for each budget level and $q$ level.

\added{All of these instances satisfy $q > 0$ or $q = 0$. Although in Section~\ref{subsec:qmixed} we derive valid inequalities for instances with a mix of positive and zero $q$ values, we do not test instances with this structure. We speculate that the path-based decomposition may perform relatively worse on instances with this structure, as the valid inequalities from Section~\ref{subsec:qmixed} include an additional relaxation of the nonconvex constraint \eqref{eq:theta1}.}

We consider the following algorithms in our tests.
\begin{itemize}[label=\raisebox{0.5ex}{\tiny$\bullet$}]
	\itemsep0em
	\item \itab{DEF:}\, Solve DEF \eqref{eq:ext} with a MIP solver  
	\item \itab{C-DEF:}\, Solve compact DEF \eqref{eq:com} with a MIP solver
	\item \itab{BEN:}\, Benders branch-and-cut algorithm on DEF \eqref{eq:ext}; refer to \eqref{eq:benders}
	\item \itab{PATH:}\, Branch-and-cut algorithm on path-based decomposition \eqref{eq:path}
\end{itemize}
DEF and C-DEF are solved using the commercial MIP solver IBM ILOG CPLEX 12.6.3. BEN is the
Benders branch-and-cut algorithm as implemented in Bodur~\etal \cite[Section~4]{bodur2016}. The branch-and-cut algorithms, BEN and PATH, were
implemented within the CPLEX solver using lazy constraint callbacks. 
Ipopt 3.12.1 was used to solve nonlinear models \eqref{eq:sep1} and \eqref{eq:sep2} in PATH. 
The computational tests were run on a 12-core machine with two 2.66GHz Intel Xeon X5650 processors and
128GB RAM, and \added{a} one-hour time limit was imposed. We allowed four threads for DEF, while all other algorithms were limited
to one thread. We used the CPLEX default relative gap tolerance of $10^{-4}$ to terminate the branch-and-bound process. The algorithms were written in Julia 0.4.5 using the Julia for Mathematical Optimization framework \cite{lubin2015}. 

In our implementation of the BEN and PATH algorithms, when generating cuts at the root LP relaxation, we used the following 
procedure to focus the computational effort on scenarios that yield cuts. After each iteration, we make a list of
scenarios for which violated cuts were found in that iteration. Then, in the next iteration, we first only attempt to identify cuts for scenarios in this list. If
successful, the list is further reduced to the subset of scenarios that yielded a violated cut. If we
fail to find any violated cuts in the current list, we re-initialize the list with all scenarios and attempt to identify violated cuts for each scenario. 

\begin{table}
	\centering
	\begin{tabular}{llrrrr}
		\toprule
		 & & \multicolumn{4}{c}{Average solve time in seconds (\# unsolved)} \\
		 \addlinespace[0.2em] \cline{3-6} \addlinespace[0.4em]
		$q$ & $b$ & DEF & C-DEF & BEN & PATH \\ \midrule
		$0.5r$ & 30 & (2)1652.1 & 23.9 & 298.5 & 54.1 \\
		& 40 & (2)2678.1 & 35.6 & 337.7 & 32.9 \\
		& 50 & (2)2242.4 & 31.7 & 355.8 & 42.8 \\
		& 60 & (2)1891.7 & 34.0 & 454.1 & 70.4 \\
		& 70 & (1)1730.0 & 27.4 & 487.4 & 63.9 \\
		& 80 & 826.0 & 14.4 & 386.7 & 27.0 \\
		& 90 & 460.3 & 6.9 & 385.0 & 26.8 \\ \midrule
		$0.1r$ & 30 & (5) & 117.5 & 438.7 & 55.3 \\
		& 40 & (5) & 466.0 & 812.3 & 301.2 \\
		& 50 & (5) & 327.9 & 787.2 & 265.3 \\
		& 60 & (5) & 638.6 & 793.7 & 194.1 \\
		& 70 & (5) & 851.6 & 983.0 & 335.0 \\
		& 80 & (5) & (1)1257.4 & (1)1106.3 & (1)891.1 \\
		& 90 & (5) & (2)2329.4 & (3)1269.4 & (2)1497.5 \\ \midrule
		$0$ & 30 & (5) & 96.1 & 487.0 & 61.7 \\
		& 40 & (5) & 196.4 & 556.8 & 177.0 \\
		& 50 & (5) & 332.3 & 663.6 & 250.8 \\
		& 60 & (5) & 369.0 & 1141.8 & 374.8 \\
		& 70 & (5) & 341.0 & 739.5 & 354.4 \\
		& 80 & (5) & 262.9 & 535.1 & 151.5 \\
		& 90 & (5) & 540.0 & 765.7 & 635.2 \\	
		\bottomrule
	\end{tabular}
	\caption{Computational results}
	\label{tab:results}
\end{table}
\begin{table}
	\centering
	\small
	\begin{tabular}{ccrrcrrr}
		\toprule
		& & \multicolumn{2}{c}{LP relaxation} & & \multicolumn{3}{c}{LP relaxation after CPLEX cuts} \\
		\addlinespace[0.2em] \cline{3-4} \cline{6-8} \addlinespace[0.4em]
		$q$ & $b$ & BEN & PATH & & C-DEF & BEN & PATH \\ \midrule
		$0.5r$ & 30 & 10.64\% & 10.64\% & & 2.38\% & 6.11\% & 6.03\% \\
		& 40 & 11.34\% & 11.35\% & & 2.80\% & 6.28\% & 6.20\% \\
		& 50 & 11.22\% & 11.23\% & & 2.40\% & 5.60\% & 5.62\% \\
		& 60 & 10.54\% & 10.55\% & & 2.07\% & 4.76\% & 4.80\% \\
		& 70 & 8.88\% & 8.89\% & & 1.83\% & 4.14\% & 4.34\% \\
		& 80 & 6.25\% & 6.25\% & & 0.97\% & 3.13\% & 3.19\% \\
		& 90 & 3.92\% & 3.91\% & & 0.43\% & 1.92\% & 2.02\% \\ \midrule
		$0.1r$ & 30 & 22.47\% & 22.49\% & & 6.72\% & 6.11\% & 6.67\% \\
		& 40 & 26.22\% & 26.18\% & & 9.35\% & 8.73\% & 8.21\% \\
		& 50 & 27.54\% & 27.48\% & & 9.38\% & 9.01\% & 9.28\% \\
		& 60 & 28.16\% & 28.08\% & & 9.93\% & 8.75\% & 9.35\% \\
		& 70 & 28.92\% & 28.83\% & & 9.67\% & 9.30\% & 9.70\% \\
		& 80 & 30.88\% & 30.90\% & & 10.85\% & 11.57\% & 11.96\% \\
		& 90 & 33.07\% & 32.98\% & & 12.45\% & 15.18\% & 15.46\% \\ \midrule
		$0$ & 30 & 25.15\% & 25.30\% & & 7.07\% & 6.72\% & 6.50\% \\
		& 40 & 28.45\% & 28.55\% & & 9.62\% & 8.07\% & 9.25\% \\
		& 50 & 30.54\% & 30.78\% & & 11.02\% & 10.03\% & 11.51\% \\
		& 60 & 32.07\% & 32.45\% & & 12.37\% & 12.48\% & 14.20\% \\
		& 70 & 32.60\% & 33.30\% & & 11.21\% & 11.82\% & 14.92\% \\
		& 80 & 33.28\% & 33.99\% & & 10.43\% & 11.44\% & 15.87\% \\
		& 90 & 36.17\% & 36.97\% & & 12.89\% & 15.45\% & 20.94\% \\
		\bottomrule
	\end{tabular}
	\caption{Average root gaps before and after CPLEX cuts. The LP relaxation of BEN is the same as DEF and C-DEF.}
	\label{tab:lprelax}
\end{table}
\begin{table}
	\setlength{\tabcolsep}{4.5pt}
	\centering
	\small
	\begin{tabular}{ccrrrrcrr}
		\toprule
		& & \multicolumn{4}{c}{Branch-and-bound nodes} & & \multicolumn{2}{c}{Cut gen. time (sec.)} \\
		\addlinespace[0.2em] \cline{3-6} \cline{8-9} \addlinespace[0.4em]
		$q$ & $b$ & DEF & C-DEF & BEN & PATH & & BEN & PATH \\ \midrule
		0.5r & 30 & (2)1080.7 & 1406.6 & 39800.8 & 51512.2 & & 266.5 & 18.9 \\
		& 40 & (2)1982.7 & 2470.2 & 21909.8 & 7697.8 & & 313.6 & 21.2 \\
		& 50 & (2)1996.3 & 2516.4 & 24112 & 17161 & & 324.8 & 21.1 \\
		& 60 & (2)1723.7 & 3270.6 & 35094.2 & 35860.2 & & 404.5 & 22 \\
		& 70 & (1)1464 & 2242.4 & 24387.2 & 34456.2 & & 456.9 & 22.1 \\
		& 80 & 727 & 1235.4 & 5989.4 & 3509.4 & & 378.7 & 21.2 \\
		& 90 & 556 & 372 & 4930.2 & 7512.2 & & 379.1 & 17.3 \\ \midrule
		0.1r & 30 & (5) & 3870.8 & 6576.2 & 10156.8 & & 417.2 & 26.2 \\
		& 40 & (5) & 17798.2 & 87973.4 & 108169.6 & & 583.9 & 29.9 \\
		& 50 & (5) & 8950.4 & 89062.4 & 70182 & & 507.6 & 29.7 \\
		& 60 & (5) & 22457.8 & 95993.8 & 53885.8 & & 528.7 & 30.9 \\
		& 70 & (5) & 28358.4 & 176635.4 & 83490.6 & & 448.8 & 32.4 \\
		& 80 & (5) & (1)45951.8 & (1)163500.3 & (1)242734 & & (1)497.9 & (1)33.9 \\
		& 90 & (5) & (2)101616.7 & (3)170902.5 & (2)327946.3 & & (3)474.8 & (2)37.5 \\ \midrule
		0 & 30 & (5) & 2205 & 5613.4 & 9114.8 & & 462 & 27.2 \\
		& 40 & (5) & 3821.6 & 13803.4 & 29569.6 & & 488.2 & 28.1 \\
		& 50 & (5) & 6037.8 & 30095.8 & 47806.4 & & 511.4 & 28.2 \\
		& 60 & (5) & 5144.6 & 127214 & 65936 & & 546.7 & 27.9 \\
		& 70 & (5) & 5379 & 45454 & 50358.6 & & 477.9 & 28 \\
		& 80 & (5) & 2667 & 18581.6 & 11890.6 & & 406.7 & 28.1 \\
		& 90 & (5) & 7289.4 & 52156.2 & 64416.2 & & 424.5 & 29.5 \\
		\bottomrule
	\end{tabular}
	\caption{Average branch-and-bound nodes and cut generation time}
	\label{tab:bnb}
\end{table}

Table~\ref{tab:results} reports the average computational time, in seconds, over the five instances for each combination
of the vector $q$ and budget level $b$. The average times include only instances solved within the time limit -- the
number of unsolved instances are enclosed in parentheses. We find that PATH and C-DEF are consistently faster than both DEF and BEN. 
PATH and C-DEF have comparable performance, with C-DEF being somewhat faster on the $q=0.5r$ instances
(the easiest set) and PATH being somewhat faster on the $q=0.1r$ instances (the hardest set). 

Table~\ref{tab:lprelax} reports the root gap after adding LP cuts for the decomposition algorithms BEN and PATH,
relative to the optimal solution value, before and after the addition of cuts from CPLEX. The root gap is calculated as $(z^* - z^{LP}) / z^*$, where $z^*$ is the
instance's optimal value and $z^{LP}$ is the objective value obtained after running the initial cutting plane loop on
the relaxed master problem. The reported value is the arithmetic mean over the five instances for each combination of
interdicted traversal probabilities $q$ and budget level $b$. DEF, C-DEF, and BEN all have the same LP relaxation value,
so we show this gap only for BEN. The columns under the heading ``LP relaxation after CPLEX cuts'' refer to the gap that
is obtained at the root node, after 
CPLEX finishes its process of adding general-purpose cuts. We exclude DEF from these results since in many instances the
root node has not completed processing in the time limit.
These results show that the initial LP relaxation value obtained using the cuts in the PATH method is very similar to
that obtained using the BEN (or the DEF or C-DEF formulations).
We also find that the general purpose cuts in CPLEX significantly improve the relaxation value for all formulations,
with the greatest improvement coming in the C-DEF formulation. The gaps obtained after the addition of CPLEX cuts are
quite similar between PATH and BEN, with the exception of the $q=0$ instances, where the CPLEX cuts seem to be more
effective for the BEN formulation.

Table~\ref{tab:bnb} displays the average number of branch-and-bound nodes of each algorithm and the time spent generating cuts in BEN and PATH. This includes time spent generating cuts during the initial cutting plane loop as well as within the branch-and-cut process. The mean is over instances that solved within the time limit.
Consistent with the results on root relaxation gaps, we find that C-DEF requires the fewest number of branch-and-bound
nodes, and that PATH and BEN methods require a similar number of branch-and-bound nodes.
On the other hand, significantly less time is spent generating cuts in the PATH method than in the BEN method,
which explains the better performance of PATH.
In particular, generating cuts in PATH requires solving a maximum-reliability path problem for each destination node, as
opposed to the Benders approach which requires solving a linear program for each scenario. 



\section{Conclusion}

We have proposed two new methods for solving maximum-reliability SNIP. The first
method is to solve a more compact DEF, and the second is based on a reformulation derived from considering the
reliability for each path separately. Both of these methods outperform state-of-the-art methods on a class of SNIP
instances from the literature. An interesting direction for future research is to investigate whether ideas similar to the path-based
formulation might be useful for solving different variants of SNIP, such as the
maximum-flow network interdiction problem.

\bibliographystyle{plain}
\bibliography{refs.bib}

\end{document}